\documentclass[12pt,reqno]{amsproc}
\usepackage{graphicx, nicefrac, color}
\usepackage[margin=0.8in]{geometry}
\usepackage[cmtip,arrow]{xy}  
\usepackage[svgnames]{xcolor}
\usepackage{pb-diagram,pb-xy}

\usepackage{amssymb, amsmath, amsthm,eulervm}
\usepackage{hyperref}
\hypersetup{
    colorlinks,%
    citecolor=blue,%
    filecolor=black,%
    linkcolor=blue,%
    urlcolor=blue     
 }

\newtheorem{thm}{Theorem}[section]
\newtheorem{lem}[thm]{Lemma}
\newtheorem{defn}[thm]{Definition}
\newtheorem{cor}[thm]{Corollary}
\newtheorem{prop}[thm]{Proposition}

\newtheorem{rem}[thm]{Remark}

\newcommand{\setof}[1]{\left\{ {#1}\right\}}
\newcommand{\mvmap}{\Longrightarrow}

\newcommand{\supp}[1]{\left\lfloor {#1}\right\rfloor}
\newcommand{\capp}[1]{\left\lceil {#1}\right\rceil}

\newcommand{\Tub}{\text{Tub}}

\newcommand{\inv}{^{-1}}

\newcommand{\bv}{{\bf v}}

\DeclareMathOperator{\HG}{H}

\newcommand{\N}{{\mathbb{N}}}
\newcommand{\R}{{\mathbb{R}}}

\newcommand{\U}{{\mathbb{U}}}

\newcommand{\cA}{{\mathcal A}}
\newcommand{\cB}{{\mathcal B}}

\newcommand{\cL}{{\mathcal L}}

\newcommand{\cU}{{\mathcal U}}
\newcommand{\cV}{{\mathcal V}}

\newcommand{\cX}{{\mathcal X}}
\newcommand{\cY}{{\mathcal Y}}
\newcommand{\cZ}{{\mathcal Z}}

\def\setof#1{\left\{{#1}\right\}}

\def\w#1{\mbox{#1}}

\newcommand{\heta}{\gamma}

\newcommand{\ANR}{\mathbf{ANR}}

\begin{document}

\title{Reconstructing Functions from Random Samples} 

\author{Steve Ferry}
\address{Department of Mathematics, Rutgers University, 
Piscataway, NJ  08854, USA}
\email{sferry@math.rutgers.edu}

\author{Konstantin Mischaikow}
\address{Department of Mathematics, Rutgers University, 
Piscataway, NJ  08854, USA}
\email{mischaik@math.rutgers.edu}

\author{Vidit Nanda}
\address{Department of Mathematics, The University of Pennsylvania, 
Philadelphia, PA  19104, USA}
\email{vnanda@sas.upenn.edu}

\date{}

\begin{abstract}
From a sufficiently large point sample lying on a compact Riemannian submanifold of Euclidean space, one can construct a simplicial complex which is homotopy-equivalent to that manifold with high confidence. We describe a corresponding result for a Lipschitz-continuous function between two such manifolds. That is, we outline the construction of a simplicial map which recovers the induced maps on homotopy and homology groups with high confidence using only finite sampled data from the domain and range, as well as knowledge of the image of every point sampled from the domain. We provide explicit bounds on the size of the point samples required for such reconstruction in terms of intrinsic properties of the domain, the co-domain and the function. This reconstruction is robust to certain types of bounded sampling and evaluation noise.
\end{abstract}

\maketitle


\section{Introduction} \label{sec:intro}
	
The use of algebraic topological methods for the analysis of nonlinear data has become a subject of considerable interest with a wide variety of promising applications \cite{carlsson:09, edelsbrunner:harer:10, ghrist:08, KMM04}. With modern technology, large and high-dimensional datasets are easily collected. However, in many cases the data are generated by a nonlinear system with many fewer degrees of freedom than the ambient dimension, and thus one may expect that the data actually lie on a much lower dimensional manifold. In this setting, it seems reasonable that the geometry generated by the data may provide insight concerning the original system. Given that such data are typically finite and noisy, relatively crude invariants -- such as homology or homotopy groups -- appear to be natural tools for capturing aspects of their underlying topological structure.  The same heuristic arguments extend to transformations from one dataset to another. In such analysis, one is interested not only in the geometry associated with the data, but also in the action of an unknown nonlinear process on the data. This action may be partially characterized by the maps which it induces on homology and homotopy groups.

Given these arguments, an obvious but mostly unresolved mathematical issue is to quantify the extent to which  we can extract correct topological features from noisy data. An important first step in this direction is the following result due to Partha Niyogi, Steve Smale, and Shmuel Weinberger \cite{NSW06}. Let $\R^n$ denote Euclidean $n$-dimensional space with the standard metric, and consider a compact $k$-dimensional Riemannian submanifold $\cX \subset \R^n$.  The {\em condition number} $\nicefrac{1}{\tau_\cX}$ of $\cX$ is defined as follows: $\tau_\cX$  is the largest positive real number such that for any $r \in (0,\tau_\cX)$, the normal bundle of radius $r$ about $\cX$ can be embedded in $\R^n$. 

\begin{thm}
\label{thm:nswmain}
{\em (\cite[Thm 3.1]{NSW06})}
{
Let $\cX$ be a compact $k$-dimensional Riemannian submanifold of $\R^n$ with condition number $\nicefrac{1}{\tau_\cX}$. Given 
\begin{enumerate}
\item some probability parameter $\delta \in (0,1]$, 
\item a radius $\epsilon < \nicefrac{\tau_\cX}{2}$, and 
\item a finite set $X \subset \cX$ of independent and identically distributed (i.i.d.) uniformly sampled points,
\end{enumerate}
let $\cU(X)$ denote the union of $n$-dimensional open $\epsilon$-balls centered at the points in $X$. If the sample size $\#X$ is larger than a bounding value $\beta_\cX(\epsilon,\delta)$ (see Definition~\ref{defn:boundingfunction}), then $\cU_\epsilon(X)$ is homotopy-equivalent to $\cX$ with probability exceeding $(1-\delta)$.
}
\end{thm} 
Thus, the union of balls of a suitably chosen radius around a sufficiently large point sample suffices to recover the homotopy type of that manifold with high confidence. From a computational perspective, recall that the {\em nerve} of a cover \cite{munkres,spanier} is the abstract simplicial complex where each $d$-dimensional simplex corresponds to an intersection of $d+1$ sets of that cover. If we let $N(X)$ denote the nerve corresponding to the cover of $\cU(X)$ by its constituent open balls, then one obtains an isomorphism $\HG_*(\cX) \simeq \HG^\Delta_*(N(X))$ between the singular homology of $\cX$ and the simplicial homology of $N(X)$ by using the {\em nerve lemma}. Thus, it is possible to successfully compute the homology of an unknown manifold $\cX$ with high probability from a sufficiently dense point sample $X \subset \cX$. 

An obvious next step is to obtain bounds on the probability of reconstructing, up to homotopy, a continuous function between Riemannian manifolds from images of dense samples. This is the focus of our work with the main result as follows.
\begin{thm}
\label{thm:finalres}
Let $\cX \subset \R^n$ and $\cY \subset \R^m$  be compact Riemannian submanifolds with condition numbers $\nicefrac{1}{\tau_{\cX}}$ and $\nicefrac{1}{\tau_{\cY}}$ respectively and let $f:\cX \to \cY$ be a Lipschitz continuous function with Lipschitz constant bounded above by some $\kappa \geq 0$. Given
\begin{enumerate}
\item probability parameters $\delta_\cX,\delta_\cY \in (0,1]$,
\item radii $\epsilon_\cX < \nicefrac{\tau_\cX}{2}$ and $\epsilon_\cY < \nicefrac{\tau_\cY}{2}$ satisfying $4\kappa\cdot\epsilon_\cX < \epsilon_\cY$, and
\item finite sets $X \subset \cX$ and $Y \subset \cY$ of independent and identically distributed (i.i.d.) uniformly sampled points,
\end{enumerate}
let $N(X)$ and $N(Y)$ be nerves of the covers generated by open balls of radius $\epsilon_\cX$ and $\epsilon_\cY$ around $X$ and $Y$ respectively. If $\#X > \beta_\cX(\epsilon_\cX, \delta_\cX)$ and $\#Y > \beta_\cY(\epsilon_\cY, \delta_\cY)$, then there exists a simplicial map $\phi:N(X) \to N(Y)$ which 
\begin{enumerate}
\item recovers the homotopy class of $f$ with probability exceeding $(1-\delta_\cX)(1-\delta_\cY)$, and
\item can be explicitly constructed using only $X$, $Y$, $\epsilon_\cX$, $\epsilon_\cY$, $\kappa$ and the restriction $f|_X$.
\end{enumerate}
\end{thm}
In particular, the morphisms induced by $\phi$ on the simplicial homology as well as homotopy groups of the nerves faithfully capture their singular counterparts induced by $f$ with high confidence. Note from Theorem \ref{thm:nswmain} that the upper bound on the probability of failing to produce $\phi$ is no larger than the corresponding bound on failing to reconstruct both $\cX$ and $\cY$ from the sample sets $X$ and $Y$. 

The third hypothesis in Theorem~\ref{thm:nswmain} which requires the sampled points to lie on the underlying manifold is too strong from the perspective of  practical sampling considerations.  A more reasonable hypothesis is that the data is {\em noisy} and lies near -- rather than on -- the respective underlying manifolds. This situation is also considered in \cite{NSW06}. Using their framework, we show that Theorem~\ref{thm:finalres} can be extended to this more general setting (see Theorem~\ref{thm:finalresnoise}).

The rest of the paper is organized as follows. In Section \ref{sec:nercar} we mention relevant definitions and tools from combinatorial algebraic topology. Section \ref{sec:defs} describes  results from \cite{NSW06} that are used in our work. The simplicial reconstruction of functions and its verification are presented in Section \ref{sec:proofs}. Section \ref{sec:noise} demonstrates the robustness of this reconstruction by providing a version of Theorem \ref{thm:finalres} which holds in the case of bounded sampling and evaluation noise. Our argument relies on a controlled version of the nerve lemma whose proof is described in Section \ref{sec:blocknerve}.

\section{Carriers and Nerves}\label{sec:nercar}
	 For the sake of completeness and to introduce relevant notation, we review some classical results from the theory of simplicial complexes. A much more complete treatment is available in standard texts, for instance \cite{munkres,spanier}.

Let $\U$ be any finite set whose elements we call {\em vertices}. A {\em simplicial complex} $K$ with vertex set $\U$ is a collection of nonempty subsets of $\U$ -- called {\em simplices} -- which contains all the vertices is closed under inclusion. More precisely, a collection $K$ of subsets of $\U$ is a simplicial complex if
\begin{itemize}
\item for each $u \in \U$, we have $\setof{u} \in K$, and
\item if $\sigma \in K$ and $\tau \subset \sigma$, then $\tau \in K$.
\end{itemize}
The {\em dimension} of each simplex $\sigma \in K$ is the natural number given by $\dim \sigma = \#\sigma - 1$ where $\#$ indicates cardinality. A {\em subcomplex} of $K$ is a sub-collection of simplices which forms a simplicial complex in its own right. The notation $K^d$ is used to indicate all simplices of dimension $d$ in $K$. Without loss of generality, we may identify $K^0$ with $\U$ by associating each $0$-dimensional simplex with the unique element of $\U$ which it contains. Thus, each simplex is uniquely determined by its constituent vertices. Given simplicial complexes $K$ and $L$, a {\em simplicial map} $\psi:K \to L$ associates to each vertex $v \in K^0$ a vertex $\psi(u) \in L^0$ so that for each simplex $\sigma \in K$ the image $\psi(\sigma)$ is a simplex in $L$.

\subsection{Carriers}

The {\em geometric realization} $|K|$ of a simplicial complex $K$ is defined (in \cite[Ch. 3.1]{spanier}, for instance) as the space of all maps $\alpha:K^0 \to [0,1]$ called {\em barycentric functions} such that for each $\alpha \in |K|$,
\begin{enumerate}
\item there exists $\sigma \in K$ with $\setof{u \in K^0 \mid \alpha(u) \neq 0} = \sigma$, and
\item the sum $\sum_{u \in K^0}\alpha(u)$ equals $1$. 
\end{enumerate}
The realization of a simplex $\sigma \in K$ is defined as the closed subset $|\sigma| \subset |K|$ consisting of all barycentric functions $\alpha \in |K|$ such that $\alpha(u) = 0$ whenever $u \notin \sigma$. Observe that if $\sigma \subset \tau$ then $|\sigma| \subset |\tau|$; if $\sigma \in K$ then $|\sigma|$ is contractible in $|K|$; and $|K| = \bigcup_{\sigma \in K}|\sigma|$.

A simplicial map $\psi: K \to L$ induces a continuous function $|\psi|:|K| \to |L|$ between geometric realizations defined as follows. For any $\alpha \in |K|$, the action of the barycentric function $|\psi|(\alpha) \in |L|$ on a vertex $v \in L^0$ is given by
\[
|\psi|(\alpha)(v) = \sum_{\psi(u)=v}\hspace{-1em}\alpha(u).
\]
It is readily seen from this definition that $|\psi|(|\sigma|) \subset |\psi(\sigma)|$ for each simplicial map $\psi:K \to L$ and each $\sigma \in K$. 

Let $\cU$ be any topological space and $K$ a simplicial complex. A {\em contractible carrier} $C:K \mvmap \cU$ from $K$ to $\cU$ assigns to each simplex $\sigma \in K$ a contractible subset $C(\sigma)$ of $\cU$ so that $C(\sigma) \subset C(\tau)$ whenever $\sigma \subset \tau$. A function $h:|K| \to \cU$ is {\em carried} by $C$ if $g(|\sigma|) \subset C(\sigma)$ for each simplex $\sigma \in K$. The following result is an extremely useful tool in combinatorial algebraic topology. We refer the reader to \cite{bjorner03} and the references therein for details.

\begin{lem}[\bf Carrier Lemma]
\label{lem:car}
Let $K$ be a simplicial complex, $\cU$ a topological space and $C:K \mvmap \cU$ a contractible carrier. Then, there exists a continuous function from $|K|$ to $\cU$ carried by $C$. Moreover, if two continuous functions $h, h':|K| \to \cU$ are carried by $C$, then
\begin{enumerate} 
\item they are homotopic, i.e., $h \sim h'$, and 
\item a homotopy $\Theta:|K| \times [0,1] \to S$ may be chosen so that for each $t \in [0,1]$ the section $\Theta(\ast,t):|K| \to \cU$ is also carried by $C$.
\end{enumerate}
\end{lem}

\subsection{Nerves}

Let $\cU$ be a topological space equipped with a finite cover $\U$ consisting of subsets of $M$. The {\em nerve} of $\U$ is the simplicial complex $N(\U)$ with vertex set $\U$ where each subcollection $\sigma \subset \U$ constitutes a simplex if and only if the intersection $\bigcap_{u \in \sigma}u$ is a non-empty subset of $\cU$. We call this intersection the {\em support} of $\sigma$ and denote it by $\supp{\sigma} \subset \cU$. On the other hand, we also make use of the union $\capp{\sigma} = \bigcup_{u \in \sigma}u \subset \cU$. When considering an entire subcomplex $K \subset N(\U)$ rather than a single simplex, one defines $\capp{K}$ to be the union of supports of all simplices in $K$. It is easy to see\footnote{Note that the supports of $0$-simplices are maximal among all supports in the partial order induced by inclusion.} that $\capp{K} = \bigcup_{u \in K^0}u$, and that $\capp{N(\U)} = \cU$.

A cover $\U$ of a topological space $\cU$ is called {\em contractible} if the support $\supp{\sigma}$ of each $\sigma \in N(\U)$ is a contractible subset of $\cU$. For instance, if $\cU$ lies in a topological vector space and if each $u \in \U$ is convex, then all non-empty intersections are automatically convex and hence contractible. One reason to consider contractible nerves is the following classical result (see \cite{borsuk} or \cite[Thm. 15.21]{kozlov08}).
\begin{lem}[\bf Nerve Lemma]
{
Let $\cU$ be a paracompact topological space equipped with an open cover $\U$. If $\U$ is contractible, then the geometric realization $|N(\U)|$ of its nerve is homotopy-equivalent to $\cU$.
}
\end{lem}

Since we always restrict our attention to finite covers by open balls in Euclidean space, we may obtain the following {\em controlled} version of the nerve lemma by strengthening its hypotheses.

\begin{lem}
\label{lem:contner}
Let $\U$ be a finite collection of open balls in Euclidean space and let $\cU$ be their union. Then, 
\begin{enumerate}
\item $|N(\U)|$ is homotopy-equivalent to $\cU$, and
\item a homotopy-equivalence $\zeta:|N(\U)| \to \cU$ may be chosen so that $\zeta(|\sigma|) \subset \capp{\sigma}$ for each simplex $\sigma \in N(\U)$.
\end{enumerate}
\end{lem}

Note that the first conclusion of Lemma \ref{lem:contner} follows easily from the traditional nerve lemma; it is the second assertion which plays a fundamental role in the proofs of our main results. A detailed verification of the controlled nerve lemma is presented in Section \ref{sec:blocknerve}.
      
\section{Recovering Manifolds from Samples}\label{sec:defs}
         Our main result focuses on recovering -- up to homotopy type -- a Lipschitz continuous function between finitely sampled unknown compact Riemannian submanifolds of Euclidean space. In order to accomplish this, we first use the finite sampled data to construct a simplicial complex homotopically faithful to the underlying manifold. In this section, we briefly survey the process from \cite{NSW06} which constructs such a simplicial complex.

Throughout this section, let $\cX \subset \R^n$ be a compact Riemannian submanifold with condition number $\nicefrac{1}{\tau_\cX}$.
\begin{defn}
\label{defn:boundingfunction}
{\em
The {\em bounding function} $\beta_\cX:\R^+ \times (0,1] \to \R$ is given by
\begin{align}
\label{eqn:boundfunc}   
 \beta_\cX(\epsilon, \delta) = \beta_1\left[\log(\beta_2) + \log(\nicefrac{1}{\delta}) \right], 
\end{align} 
where
\begin{align*}
\beta_1 &= \frac {\w{vol}(\cX)} {\cos^k \left( \arcsin(\nicefrac{\epsilon}{8\tau_{\cX}}) \right)\cdot\w{vol}(\cB^k_{\epsilon/4})}, \\
 \beta_2 &= \frac{\w{vol}(\cX)}{\cos^k \left( \arcsin(\nicefrac{\epsilon}{16\tau_{\cX}}) \right)\cdot\w{vol}(\cB^k_{\epsilon/8})},
\end{align*}
and $\w{vol}(\cB^k_\epsilon)$ denotes the volume of the standard $k$-dimensional Euclidean ball of radius $\epsilon$.
}
\end{defn}

Let $\cB_r(x)$ denote the $n$-dimensional Euclidean open ball of radius $r$ centered at $x \in \R^n$. For any subset $P$ of $\R^n$ and $\alpha > 0$, we denote by $\U_\alpha(P)$ the set of open balls $\setof{\cB_\alpha(p)\mid p \in P}$ and let $\cU_\alpha(P)$ be their union. We say that $P$ is {\em $\alpha$-dense} in the manifold $\cX$ if we have the inclusion $\cX \subset \cU_\alpha(P)$. The following proposition enables one to recover the homotopy type of $\cX$ from a finite set $X$ which is sufficiently dense in $\cX$ relative to $\tau_\cX$.

\begin{prop}
\label{prop:nswdens} 
{\em (\cite[Prop 3.1]{NSW06})}
Assume $\epsilon \in \left(0,\sqrt{\nicefrac{3}{5}}~\tau_\cX\right)$ and that a finite set $X \subset \R^n$ is $\nicefrac{\epsilon}{2}$-dense in $\cX$. Then, the canonical projection map $\pi_\cX: \cU_\epsilon(X) \to \cX$ defined by
\begin{align}
\label{eqn:pidef}
\pi_\cX(w) = \arg\min_{x \in \cX}\|w-x\|_{\R^n}
\end{align}
is a strong deformation-retraction.
\end{prop}
    
As a corollary to this proposition one obtains a string of isomorphisms on homology: 
\[
\HG_*(\cX) \simeq \HG_*\left(\cU_\epsilon(X)\right) \simeq \HG_*(|N_\epsilon(X)|) \simeq \HG_*^\Delta(N_\epsilon(X)).
\]
The first isomorphism comes from the fact that deformation retractions preserve homotopy type and homology is a homotopy-invariant. The second isomorphism results from applying the nerve lemma: since $\U_\epsilon$ is a convex cover of $\cU_\epsilon(X)$ for each $\epsilon$, the associated nerve is contractible. The last isomorphism is simply the equivalence of singular and simplicial homology. We remark here that the nerve of a cover by balls \cite{cechcomp} and the homology groups of finite simplicial complexes \cite{gaphom, HMMN} are eminently computable. Thus, one can actually obtain a finite representation of the homology of $\cX$ up to isomorphism from a sufficiently dense point sample. 

Adopting the terminology of the proposition, we observe an important property of the projection map $\pi_\cX:\cU_\epsilon(X) \to \cX$ which holds even when $\epsilon$ is allowed to range over the larger interval $(0,\tau_\cX)$. Note that for each $w \in \cU_\epsilon(X)$ we have some $\xi \in X$ with $\|\xi-w\|_{\R^n} < \epsilon$. On the other hand, the distance between $w$ and $\pi_\cX(w)$ is at most $\epsilon$ as well, since $\pi_\cX(w)$ is the nearest point of the manifold $\cX$ to $w$. By the triangle inequality, one has the following estimate for each $\xi \in X$ and $w \in \cB_\epsilon(\xi)$ whenever $\epsilon_\cX \in (0,\tau_\cX)$:
\begin{align}
\label{eqn:pidist}
\|\pi_\cX(w) - \xi\|_{\R^n} < 2\epsilon_\cX.
\end{align}
This is not the best possible estimate one could obtain, but it suffices for our purposes here.

The following proposition  assumes that $X$ is obtained by uniform i.i.d.\ sampling on $\cX$ and provides a lower bound on the sample size $\#X$ which guarantees -- with high confidence -- the $\nicefrac{\epsilon}{2}$-density needed by the previous proposition. 

\begin{prop}
\label{prop:nswprob}
{\em (\cite[Prop 3.2]{NSW06})}
Choose $\epsilon \in \left(0,\nicefrac{\tau_\cX}{2}\right)$ and the probability parameter $\delta \in (0,1]$. Assume that $X$ is obtained by i.i.d.\ uniform samplings from $\cX$. If $\#X > \beta_\cX(\epsilon,\delta)$, then $X$ is $\nicefrac{\epsilon}{2}$-dense in $\cX$ with probability exceeding $(1-\delta)$.
\end{prop}

Propositions \ref{prop:nswdens} and \ref{prop:nswprob} lead directly to Theorem \ref{thm:nswmain}, which is the main result of \cite{NSW06}.

\section{Recovering Functions from Samples}\label{sec:proofs}
	 In this section we provide a proof of our main result, Theorem~\ref{thm:finalres}. The hypotheses of this theorem consist of a variety of assumptions and a-priori choices of parameters. To clarify their respective roles we present them via the following exhaustive list. The notation below will remain fixed throughout this section.

\begin{itemize}
\item[\bf Cnd:] Assume that $\cX \subset \R^n$ and $\cY \subset \R^m$ are compact Riemannian submanifolds with condition numbers $\nicefrac{1}{\tau_\cX}$ and $\nicefrac{1}{\tau_\cY}$, respectively.
\item[\bf Prb:] Choose probability parameters $\delta_\cX, \delta_\cY \in (0,1]$. 
\item[\bf Lip:] Assume $f:\cX\to \cY$ is a Lipschitz continuous function with Lipschitz constant less than $\kappa \geq 0$. More precisely, we require $\|f(x)-f(x')\|_{\R^m} \leq \kappa \|x - x'\|_{\R^n}$ for any pair of points $x,x' \in \cX$.
\item[\bf Rad:] Choose the radii $\epsilon_\cX \in (0, \nicefrac{\tau_\cX}{2})$ and $\epsilon_\cY \in (0, \nicefrac{\tau_\cY}{2})$ so that
$4\kappa\epsilon_\cX < \epsilon_\cY$.
\item [\bf Smp:]  Assume knowledge of the finite sets $X \subset \cX$ and $Y \subset \cY$ obtained by i.i.d.\ uniform sampling from $\cX$ and $\cY$ respectively. Furthermore, require $\#X > \beta_\cX(\epsilon_\cX, \delta_\cX)$ and $\#Y > \beta_\cY(\epsilon_\cY, \delta_\cY)$. 
\item [\bf Img:] Assume knowledge of the restriction $f|_X:X \to \cY \subset \R^m$ of $f$ to the point sample $X$.
\end{itemize}

It is important to note that neither {\bf Smp} nor {\bf Img} imply that sampled points  map to sampled points, so in general $f(X) \not\subset Y$. Since {\bf Rad} fixes choices of $\epsilon_\cX$ and $\epsilon_\cY$, we simplify our notation by declaring that $N(X) := N_{\epsilon_\cX}(X)$ and $N(Y) := N_{\epsilon_\cY}(Y)$. Similarly, we denote the unions of balls $\cU_{\epsilon_\cX}(X)$ and $\cU_{\epsilon_\cY}(Y)$ by $\cU(X)$ and $\cU(Y)$. Finally, define $\rho \in \R$ by
\begin{equation}
\label{eq:rho}
\rho =  \epsilon_\cY - 2\kappa\epsilon_\cX.
\end{equation}
Note that by {\bf Rad}, we have $\rho > \nicefrac{\epsilon_\cY}{2} > 0$.

By {\bf Cnd}, {\bf Prb}, {\bf Rad} and {\bf Smp}, Theorem \ref{thm:nswmain} establishes that the vertical maps in the diagram below induce isomorphisms on homotopy with probability exceeding $(1-\delta_\cX)(1-\delta_\cY)$:
\[
\begin{diagram}
\dgARROWLENGTH .35in
\node{\cX} \arrow{e,t}{f}
\node{\cY} \arrow{s,r}{\iota_\cY}
\\
\node{\cU(X)} \arrow{n,l}{\pi_\cX}\arrow{e,b,..}{g}
\node{\cU(Y)}
\end{diagram}
\]
Here, $\pi_\cX$ is the canonical projection map (\ref{eqn:pidef}) and $\iota_\cY$ is the inclusion of $\cY$ into the union of balls $\cU(Y)$. Let $g:\cU(X) \to \cU(Y)$ be the composition $\iota_\cY \circ f \circ \pi_\cX$. With probability exceeding $(1-\delta_\cX)(1-\delta_\cY)$, the map induced by $g$ on homotopy is naturally related by isomorphisms to the corresponding map induced by $f$.

\begin{defn}
{\em
A {\em simplicial reconstruction} of $f$ is defined to be any simplicial map $\phi:N(X) \to N(Y)$ so that $g(\capp{\sigma}) \subset \capp{\phi(\sigma)}$ for each $\sigma \in N(X)$.
}
\end{defn}

We know from Lemma \ref{lem:contner} that it is possible to find continuous maps $\zeta_Z:|N(Z)| \to \cU(Z)$ for $Z \in \setof{X,Y}$ which induce isomorphisms on homotopy and satisfy $\zeta_Z(|\sigma|) \subset \capp{\sigma}$ for each simplex $\sigma \in N(Z)$. For any choice of such maps, the next proposition establishes that the following diagram commutes up to homotopy whenever $\phi$ is a simplicial reconstruction of $f$:
\[
\begin{diagram}
\dgARROWLENGTH .35in
\node{\cU(X)} \arrow{e,b}{g}
\node{\cU(Y)}
\\
\node{|N(X)|} \arrow{e,t,..}{|\phi|}\arrow{n,l}{\zeta_X} 
\node{|N(Y)|} \arrow{n,r}{\zeta_Y}
\end{diagram}
\]

\begin{prop}
\label{prop:rechom}
If $\phi$ is a simplicial reconstruction of $f$, then $\zeta_Y \circ |\phi|$ and $g \circ \zeta_X$ share a contractible carrier and hence are homotopic.
\end{prop}
\begin{proof}
For each $\sigma \in N(X)$, define $F(\sigma) = \capp{\phi(\sigma)}$. Note that $\sigma \subset \sigma'$ implies $F(\sigma) \subset F(\sigma')$ since the latter is the union over a superset. The $F$-image of each $\sigma \in N(X)$ is contractible in $\cU(Y)$, being a union of convex sets (in our case, balls in Euclidean space) with a non-empty intersection. Thus, $F:N(X) \mvmap \cU(Y)$ is a contractible carrier. Given any $\sigma \in N(X)$, we have
\begin{align*}
\zeta_Y \circ |\phi|(|\sigma|) & \subset \zeta_Y(|\phi(\sigma)|) \text{, since $\phi$ is a simplicial map, }\\
															 & \subset \capp{\phi(\sigma)} \text{, by assumption on $\zeta_Y$,}\\
															 & = F(\sigma).
\end{align*}
On the other hand,
\begin{align*}
g \circ \zeta_X(|\sigma|) & \subset g\left(\capp{\sigma}\right)\text{, by assumption on $\zeta_X$, }\\
													& \subset \capp{\phi(\sigma)} \text{, since $\phi$ is a simplicial reconstruction of $f$,} \\
													& = F(\sigma).
\end{align*}
Thus, both $\zeta_Y \circ |\phi|$ and $g \circ \zeta_X$ are carried by $F$.
\end{proof}

With the goal of building a simplicial reconstruction in mind, we consider the correspondence $\Delta: X \to 2^Y$ defined on each $\xi \in X$ by
\begin{align}
\label{eqn:Ldef}
\Delta(\xi) = \setof{\eta \in Y \text{ so that } \|\eta - f(\xi)\|_{\R^m} < \rho}.
\end{align}

\begin{prop}
\label{prop:nonempty}
With probability exceeding $(1 - \delta_\cY)$, the following holds. For each $\xi \in X$, the set $\Delta(\xi) \subset Y$ is non-empty.
\end{prop}
\begin{proof}
Since {\bf Cnd}, {\bf Prb}, {\bf Rad} and {\bf Smp} satisfy the hypotheses of Proposition \ref{prop:nswprob} for $\cY$, we see that $Y$ is $\nicefrac{\epsilon_\cY}{2}$-dense in $\cY$ with probability exceeding $(1-\delta_\cY)$. This density suffices to guarantee a non-empty $\Delta(\xi)$ for each $\xi \in X$ in the following way. For any $\xi \in X$ there exists some $\eta \in Y$ with $\|f(\xi) - \eta)\|_{\R^m} < \nicefrac{\epsilon_\cY}{2}$. Since $\nicefrac{\epsilon_\cY}{2} < \rho$ as a consequence of {\bf Rad} and (\ref{eq:rho}), we have $\eta \in \Delta(\xi)$. 
\end{proof}

Let $\bv_X:N(X) \to 2^X$ denote the map taking each simplex $\sigma \in N(X)$ to its vertex set $\bv_X(\sigma) \subset X$ which corresponds to (the centers of) those balls whose non-empty intersection determines the support of $\sigma$. Note that $\sigma \subset \sigma'$ in $N(X)$ if and only if we have the inclusion $\bv_X(\sigma) \subset \bv_X(\sigma')$. Define $\bv_Y:N(Y) \to 2^Y$ similarly. For each $\sigma \in N(X)$, define $\Delta_\sigma \subset Y$ by
\begin{align}
\label{eqn:delta}
\Delta_\sigma = \bigcup_{\xi \in \bv_X(\sigma)}\hspace{-.5em}\Delta(\xi).
\end{align}

\begin{prop}
\label{prop:gammasimp}
With probability exceeding $(1-\delta_\cX)$, the following holds. For each $\sigma \in N(X)$, any non-empty subset of $\Delta_\sigma$ determines a simplex in $N(Y)$.
\end{prop}
\begin{proof}
By {\bf Cnd}, {\bf Prb}, {\bf Rad} and {\bf Smp}, Proposition \ref{prop:nswprob} holds for $\cX$. Therefore, with probability exceeding $(1-\delta_\cX)$ we are guaranteed that $X$ is sufficiently dense in $\cX$ so that the canonical projection map $\pi_\cX$ from (\ref{eqn:pidef}) induces homotopy-equivalence. Assuming this density, pick  $w \in \supp{\sigma}$ and recall that $g(w) = \iota_\cY\circ f\circ \pi_\cX(w)$ by definition. Setting $x = \pi_\cX(w)$, we note from (\ref{eqn:pidist}) that $\|x - \xi\|_{\R^n} < 2\epsilon_\cX$ for each $\xi \in \bv_X(\sigma)$. By {\bf Lip}, for each such $\xi$ we have:
\begin{align}
\label{eqn:fxcont}
\|f(x) - f(\xi)\|_{\R^m} < 2\kappa\epsilon_\cX.
\end{align}
Given $\eta \in \Delta_\sigma$, by (\ref{eqn:delta}) there is some $\xi_* \in \bv_X(\sigma)$ so that $\eta \in \Delta(\xi_*)$, whence $\|\eta - f(\xi_*)\|_{\R^m} < \rho$ by \ref{eqn:Ldef}. Since (\ref{eqn:fxcont}) implies $\|f(x) - f(\xi_*)\|_{\R^n} < 2\kappa\epsilon_\cX$, the triangle inequality yields 
\[
\|f(x) - \eta\|_{\R^m} < \rho + 2\kappa\epsilon_\cX < \epsilon_\cY.
\]
Since the point $f(x)$ lies in the intersection $\bigcap_{\eta \in \Delta_\sigma}\cB_{\epsilon_\cY}(\eta)$, this intersection is non-empty and must determine a simplex of $N(Y)$. Clearly, any subset of $\Delta_\sigma$ determines a face of this simplex, and hence constitutes a simplex in its own right.
\end{proof}

Proposition \ref{prop:nonempty} guarantees with probability exceeding $(1-\delta_\cX)$ that $\Delta(\xi)$ is nonempty for each $\xi \in X$ and so we may choose a {\em selector} function $h: X \to Y$ so that $h(\xi) \in \Delta(\xi)$ for each such $\xi$. By definition of $\Delta$ and (\ref{eq:rho}), we have 
\begin{align}
\label{eqn:approx}
g(\cB_{\epsilon_\cX}(\xi)) \subset \cB_{\epsilon_\cY}(h(\xi)) \text{ for each } \xi \in X.
\end{align} 

Proposition \ref{prop:gammasimp} guarantees with probability exceeding $(1-\delta_\cY)$ that for each $\sigma \in N(X)$, the collection $\setof{h(\xi) \mid \xi \in \bv_X(\sigma)}$ determines a simplex of $N(Y)$. Therefore, with probability exceeding $(1-\delta_\cX)(1-\delta_\cY)$, there exists a map $h:X \to Y$ of points which induces a simplicial map $\phi_h:N(X) \to N(Y)$. The following proposition demonstrates that the homotopy type of the induced simplicial map is independent of the choice of $h$.                  

\begin{prop}
\label{prop:hindep}
Given any pair $h,h':X \to Y$ of selectors, the maps $|\phi_h|$ and $|\phi_{h'}|$ from $|N(X)|$ to $|N(Y)|$ are homotopic.
\end{prop}
\begin{proof}
For each $\sigma \in N(X)$, we know that $\tau_\sigma := \bv_Y\inv(\Delta_\sigma)$ is a simplex of $N(Y)$ by Proposition \ref{prop:gammasimp}. Moreover, if $\sigma \subset \sigma'$ then $\Delta_\sigma \subset \Delta_{\sigma'}$ by \ref{eqn:delta} and therefore $\tau_\sigma \subset \tau_{\sigma'}$.. For any $\sigma \in N(X)$, note that since $\phi_h(\sigma)$ and $\phi_{h'}(\sigma)$ are faces of $\tau_\sigma$, both $|\phi_h(\sigma)|$ and $|\phi_{h'}(\sigma)|$ are subsets of $|\tau_\sigma|$. Now $|\tau_\sigma|$, being the realization of a single simplex, is contractible. Therefore, $|\phi_h|$ and $|\phi_{h'}|$ share that contractible carrier which associates each $\sigma \in N(X)$ to $|\tau_\sigma| \subset |N(Y)|$.  
\end{proof}

For any selector $h:X \to Y$, the induced map $\phi_h:N(X) \to N(Y)$ is a simplicial reconstruction of $f$. To see this, note that for any $\sigma \in N(X)$ we can apply (\ref{eqn:approx}) to each element of $\bv_X(\sigma)$ in order to conclude $g(\capp{\sigma}) \subset \capp{\phi_h(\sigma)}$. Combining this with Proposition \ref{prop:rechom} concludes the proof of Theorem \ref{thm:finalres}.

\section{Robustness to Bounded Noise}\label{sec:noise}
   As is indicated in the Introduction, we would like to extend the results of the previous section to the case where data samples are
{\em noisy}. That is, we assume that the sampled points lie close to, rather than on, the underlying manifolds. Such sampling discrepancies -- aside from being ubiquitous in experimental data -- cascade into imprecise knowledge of the images under an unknown function, particularly if the evaluation of that function is also subject to some inherent measurement errors. This generalization requires a framework to describe the noise, and so we adopt the model of \cite[Sec. 7]{NSW06}. For any subset $P$ of Euclidean space $\R^n$ and any $\alpha > 0$, define the {\em tubular neighborhood of radius $\alpha$ around $P$} as follows:
\[
\Tub_\alpha(P) = \setof{x \in \R^n \mid \text{ there is some }p \in P\text{ with }\|p - x\|_{\R^n} < \alpha}.
\] 

\begin{defn}
{\em
Given a subset $P \subset \R^n$ and some $r > 0$, a probability measure $\mu$ on $\R^n$ is called {\em $r$-conditioned about $P$} if
\begin{enumerate}
\item the support of $\mu$ is contained in $\Tub_r(P)$, and
\item for each $s \in (0,r)$, there exists some constant $\Omega_s$ so that 
\[
\inf_{p \in P}\mu(\cB_s(p)) = \Omega_s > 0
\] where $\cB_s(p)$ denotes the $n$-dimensional open ball of radius $s$ about $p$.
\end{enumerate}
We write $\Omega(\mu)$ to denote the constant $\Omega_{r/2} > 0$.
}
\end{defn}

As in the noiseless case, the following fundamental results concerning sampling of manifolds are reproduced from \cite{NSW06}. Let $\cX$ be a compact Riemannian submanifold of $\R^n$ with condition number $\nicefrac{1}{\tau_\cX}$. For $r > 0$ define the functions
\[
\Gamma_\cX^\pm(r) = \frac{(\tau_\cX + r) \pm \sqrt{\tau_\cX^2 + r^2 - 6\tau_\cX r}}{2}
\]
and note that $0 < \Gamma^-_\cX < \Gamma^+_\cX$ when the quantity under the square root is strictly positive. It is straightforward to check that this positivity holds for $r < (3 - \sqrt{8})\tau_\cX$. Pick such an $r$ and assume that $X$ is a finite set lying in $\Tub_r(\cX) \subset \R^n$. For each $\alpha > 0$, let $N_\alpha(X)$ be the nerve generated by open balls of radius $\alpha$ about the points in $X$ and let $\cU_\alpha(X)$ be their union. The following result  is the noisy analogue of Proposition \ref{prop:nswdens}.

\begin{prop}
\label{prop:nswdens-noise}
{\em (\cite[Prop 7.1]{NSW06})}
Assume that $X$ is $r$-dense in $\cX$ for some $0 < r < (3-\sqrt{8})\tau_\cX$ and choose a radius $\epsilon$ satisfying $\Gamma_\cX^-(r) < \epsilon < \Gamma_\cX^+(r)$. Then, the canonical projection map $\pi_\cX:\cU_\epsilon(X) \to \cX$ as defined in (\ref{eqn:pidef}) is a strong deformation retraction.
\end{prop}

Recall that given $\nu > 0$, the {\em $\nu$-covering number of $\cX$} -- denoted $\Lambda_\nu(\cX)$ -- is defined to be the minimum possible $q \in \N$ satisfying the following property: there exists some finite set $S \subset \cX$ of cardinality $q$ such that the collection $\setof{\cB_\alpha(s) \mid s \in S}$ of $n$-dimensional open balls covers $\cX$. Given an $r$-conditioned probability measure $\mu$ about $\cX$ and a probability parameter $\delta \in [0,1)$,  define the new bounding function $\heta_\cX$ as follows:
\begin{align}
\label{eqn:noisebound}
\heta_{\cX}(\mu,\delta) = \frac{1}{\Omega(\mu)}\left(\log(\Lambda_{\nicefrac{r}{2}}(\cX)) + \log\left(\nicefrac{1}{\delta}\right)\right).
\end{align}

The next result  replaces Proposition \ref{prop:nswprob} in the setting of conditioned noise.

\begin{prop}
\label{prop:nswprob-noise}
{\em (\cite[Prop 7.2]{NSW06})}
 Assume that the real numbers $r$ and $\epsilon$ satisfy $0 < r < (3-\sqrt{8})\tau_\cX$ and  $\Gamma_\cX^-(r) < \epsilon < \Gamma_\cX^+(r)$. Let $\mu$ be any $r$-conditioned probability measure about $\cX$ and assume that a finite set $X$ is drawn from $\R^n$ in i.i.d.\ fashion with respect to $\mu$. Given a parameter $\delta \in (0,1]$, if $\#X > \heta_\cX(\mu,\delta)$
then $X$ is $r$-dense in $\cX$ with probability exceeding $(1-\delta)$.
\end{prop}

Combining the preceding propositions yields the main result of \cite{NSW06} as adapted for conditioned noise.

\begin{thm}
\label{thm:nswmainnoise}
{\em (\cite[Thm 7.1]{NSW06})}
Let $\cX \subset \R^n$ be a compact Riemannian submanifold with condition number $\nicefrac{1}{\tau_\cX}$. Fix $r \in (0,(3-\sqrt{8})\tau_\cX)$ and choose a radius $\epsilon$ satisfying $\Gamma_\cX^-(r) < \epsilon < \Gamma_\cX^+(r)$. Assume that $\mu$ is an $r$-conditioned probability measure about $\cX$ and that $X \subset \R^n$ is a finite set obtained by $\mu$-i.i.d.\ sampling. If $\#X > \heta_\cX(\mu,\delta)$ for some $\delta \in (0,1]$, then $\cU_\epsilon(X)$ strong deformation retracts onto $\cX$ with probability exceeding $(1-\delta)$.
\end{thm}

The introduction of sampling noise requires the following modifications to our assumptions and choices.
\begin{itemize}
\item[\bf Lip':] Assume that $f:\cX\to \cY$ is a Lipschitz-continuous function whose Lipschitz constant is bounded above by some $\kappa \geq 0$ satisfying $4\kappa\cdot\tau_\cX < (\sqrt{2}-1)\tau_\cY$. 
\item[\bf Nse:] Choose positive noise bounds $r_\cZ < \alpha\cdot\tau_\cZ$ where $\alpha = (3-\sqrt{8})$ and $\cZ \in \setof{\cX,\cY}$.  Assume that $\mu_\cZ$ is a $r_\cZ$-conditioned probability measure about $\cZ$.
\item[\bf Rad':] Choose radii $\epsilon_\cZ$ satisfying $\Gamma^-_\cZ(r_\cZ) < \epsilon_\cZ < \Gamma^+_\cZ(r_\cZ)$ for $\cZ \in \setof{\cX,\cY}$ so that 
\begin{equation}
\label{eq:rad'}
4\kappa\cdot(\epsilon_\cX + r_\cX) < (\epsilon_\cY - r_\cY).
\end{equation}
\item [\bf Smp':]  Assume knowledge of the finite sets $X' \subset \R^n$ and 
$Y' \subset \R^m$ obtained by i.i.d.\ $\mu_\cX$ and $\mu_\cY$ sampling respectively. We require $\# X' > \heta_\cX(\mu_\cX, \delta_\cX)$ and $\# Y' > \heta_\cY(\mu_\cY, \delta_\cY)$. 
\item [\bf Img':] Assume knowledge of $f':X' \to \R^m$ so that $\|f'(\xi) - f\circ \pi_\cX(\xi)\| \leq d$ for each $\xi \in X'$, where $\pi_\cX$ is the canonical projection map from (\ref{eqn:pidef}) and $d > 0$ satisfies the following bound:
\begin{equation}
\label{eq:d}
d < \frac{(\epsilon_\cY - r_\cY) - 2\kappa\cdot(\epsilon_\cX + r_\cX)}{2}.
\end{equation}
\end{itemize}
The assumptions {\bf Cnd} and {\bf Prb} of Section \ref{sec:proofs} remain unchanged. As usual, we simplify notation by dropping the fixed quantities $\epsilon_\cX$ and $\epsilon_\cY$ from various subscripts. Thus, the nerve $N_{\epsilon_\cX}(X')$ and the union $\cU_{\epsilon_\cX}(X')$ are denoted by $N(X')$ and $\cU(X')$ respectively, and similar simplifications are made for the $Y'$ analogues. Note that in the assumption {\bf Img'} we allow {\em evaluation noise}. That is, we only assume knowledge of the true image $f\circ\pi_\cX(\xi) \in \cY$ of each $\xi \in X'$ up to a distance of $d$. 

The inequality (\ref{eq:rad'}) is a constraint that involves the Lipschitz constant of $f$, the models for conditioned noise, and the radii for the nerves. It guarantees that  the restriction (\ref{eq:d}) is always positive. The following result provides  conditions on the manifolds, the noise models and the function under which  (\ref{eq:rad'}) can be satisfied.

\begin{prop}
\label{prop:epsilonsuff}
If $4\kappa\cdot\tau_\cX < (\sqrt{2}-1)\tau_\cY$, then there exist valid choices of $\epsilon_\cX$ and $\epsilon_\cY$ which satisfy (\ref{eq:rad'}). 
\end{prop}
\begin{proof}
First, we check that $(\epsilon_\cX + r_\cX) < \tau_\cX$ on the domain $0 < r_\cX < (3-\sqrt{8})\tau_\cX$ imposed by {\bf Nse}. Recall that $\epsilon_\cX < \Gamma^+_\cX(r_\cX)$ by {\bf Rad'}, and consider the following function
\[
2(\Gamma^+_\cX(r_\cX) + r_\cX) = (\tau_\cX + 3r_\cX) + \sqrt{\tau_\cX^2 + r_\cX^2 - 6\tau_\cX r_\cX}
\] 
This function has no local maximum in its domain and attains a maximum value of $2\tau_\cX$ at the left endpoint, so $(\epsilon_\cX+r_\cX) < \tau_\cX$ as desired. Since (\ref{eq:rad'}) imposes a lower bound of $4\kappa\cdot(\epsilon_\cX + r_\cX) + r_\cY$ on $\epsilon_\cY$, it suffices to show that the over-estimate $4\kappa\cdot\tau_\cX + r_\cY$ of this lower bound is smaller than the upper bound $\Gamma^+_\cY(r_\cY)$ imposed on $\epsilon_\cY$ by {\bf Rad'}. Equivalently, we must show that $\Gamma^+_\cY(r_\cY) - r_\cY > 4\kappa\cdot\tau_\cX$. Observe that the function
\[
\Gamma_\cY^+(r_\cY) - r_\cY = \frac{(\tau_\cY - r_\cY) + \sqrt{\tau_\cY^2 + r_\cY^2 - 6\tau_\cY r_\cY}}{2}
\]
has no local minima on the domain $(0,(3-\sqrt{8})\tau_\cY)$ imposed by {\bf Nse} and attains a minimum value of $(\sqrt{2}-1)\tau_\cY$ at the right endpoint. Thus, it is possible to satisfy (\ref{eq:rad'}) if $4\kappa\cdot\tau_\cX < (\sqrt{2}-1)\tau_\cY$.
\end{proof}

The main result of this section is the following theorem.

\begin{thm}
\label{thm:finalresnoise}
Assume {\bf Cnd}, {\bf Lip'}, {\bf Nse}, {\bf Rad'}, {\bf Prb}, {\bf Smp'}, and {\bf Img'}. Then, with probability exceeding $(1-\delta_
\cX)(1-\delta_\cY)$ there exists a simplicial map $\phi:N(X') \to N(Y')$ which
\begin{enumerate}
\item is a simplicial reconstruction of $f$, and
\item can be explicitly constructed using only $X'$, $Y'$, $\epsilon_\cX$, $\epsilon_\cY$, $\kappa$ and $f'$.
\end{enumerate}
\end{thm}

For the most part, the proof of this theorem is analogous to that of Theorem~\ref{thm:finalres}. Aside from Proposition \ref{prop:epsilonsuff}, the only major modification is that the domain's radius $\epsilon_\cX$ is augmented by the noise bound $r_\cX$ whereas the range's radius $\epsilon_\cY$ is diminished by the corresponding bound $r_\cY$. The following quantity plays the role of $\rho$ from (\ref{eq:rho}).
\begin{align}
\label{eq:rho'}
\rho' = r_\cY + d.
\end{align}
Define $\Delta':X' \to 2^{Y'}$ as follows: for each $\xi \in X'$,
\begin{align}
\label{eqn:Ldef'}
\Delta'(\xi) = \setof{\eta \in Y' \text{ so that } \|f'(\xi) - \eta\|_{\R^m} < \rho'}.
\end{align}
The noisy analogue of Proposition \ref{prop:nonempty} is as follows.
\begin{prop}
\label{prop:nonempty-noise}
With probability exceeding $(1 - \delta_\cY)$, the following holds. For each $\xi \in X'$, the set $\Delta'(\xi) \subset Y$ is non-empty.
\end{prop}
\begin{proof}
Since {\bf Cnd}, {\bf Prb}, {\bf Nse}, {\bf Rad'} and {\bf Smp'} satisfy the hypotheses of Proposition \ref{prop:nswprob-noise} for $\cY$ with probability measure $\mu_\cY$, the sampled set $Y'$ is $r_\cY$-dense in $\cY$ with probability exceeding $(1-\delta_\cY)$. Assume that this density holds. Now, $\|f'(\xi) - f\circ\pi_\cX(\xi)\|_{\R^m} < d$ by {\bf Img'} and there exists some $\eta \in Y$ with $\|f\circ\pi_\cX(\xi) - \eta\|_{\R^m} < r_\cY$ by the assumed density of $Y'$. By the triangle inequality, $\|f'(x) - \eta\| < r_\cY + d$, and hence $\eta \in \Delta'(\xi)$. 
\end{proof}

As before, we define maps $\bv_Z:N(Z) \to 2^Z$ for $Z \in \setof{X',Y'}$ taking each simplex to its vertex set. For each $\sigma \in N(X')$, define
\begin{align}
\label{eqn:delta'}
\Delta'_\sigma = \bigcup_{\xi \in \bv_{X'}(\sigma)}\hspace{-.5em}\Delta'(\xi).
\end{align}
\begin{prop}
\label{prop:gammasimp-noise}
With probability exceeding $(1-\delta_\cX)$, the following is true. For each $\sigma \in N(X')$, any non-empty subset of $\Delta'_\sigma$ determines a simplex in $N(Y')$.
\end{prop}
\begin{proof}
By {\bf Cnd}, {\bf Prb}, {\bf Rad'}, {\bf Nse} and {\bf Smp'}, Proposition \ref{prop:nswprob-noise} holds for $\cX$ and the probability measure $\mu_\cX$. Therefore, with probability exceeding $(1-\delta_\cX)$ we are guaranteed that $X'$ is $r_\cX$-dense in $\cX$ via Proposition \ref{prop:nswprob-noise} and hence that $\pi_\cX$ is a strong deformation retraction via Proposition \ref{prop:nswdens-noise}. Assuming this density, note that the distance from any $\xi \in X'$ to its nearest neighbor $\pi_\cX(\xi)$ in $\cX$ is at most $r_\cX$, since $\mu_\cX$ is $r_\cX$-conditioned about $\cX$ by {\bf Nse}. Similarly, given any $w \in \cU(X')$, we have $\|w - \pi_\cX(w)\|_{\R^n} < \epsilon_\cX+r_\cX$. For any $w \in \supp{\sigma}$ and $\xi \in \bv_{X'}(\sigma)$, we have $\|w-\xi\|_{\R^n} < \epsilon_\cX$ by definition, so we may make the following estimate:
\begin{align*}
\|\pi_\cX(w)-\pi_\cX(\xi)\|_{\R^n} &\leq \|\pi_\cX(w)-w\|_{\R^n}+\|w-\xi\|_{\R^n} + \|\xi-\pi_\cX(\xi)\|_{\R^n} \\
																	 &< (\epsilon_\cX+r_\cX) + \epsilon_\cX + r_\cX \\
																	 &= 2(\epsilon_\cX+r_\cX).
\end{align*}
Write $x = \pi_\cX(w)$, and observe by {\bf Lip'} that
\[
\|f(x) - f\circ\pi_\cX(\xi)\|_{\R^m} < 2\kappa(\epsilon_\cX + r_\cX).
\]
By (\ref{eq:d}) we know $\|f'(\xi) - f\circ\pi_\cX(\xi)\|_{\R^m} < d$, so we obtain
\begin{align}
\label{eqn:fxcont-noise}
\|f(x) - f'(\xi)\|_{\R^m} < 2\kappa(\epsilon_\cX + r_\cX) + d.
\end{align}
Pick any $\eta \in \Delta_\sigma$, so by (\ref{eqn:delta'}) there is some $\xi_* \in \bv_{X'}(\sigma)$ with $\eta \in \Delta'(\xi_*)$, whence $\|\eta - f'(\xi_*)\|_{\R^m} < \rho'$ by (\ref{eqn:Ldef'}). Since (\ref{eqn:fxcont-noise}) implies $\|f(x) - f'(\xi_*)\|_{\R^m} < 2\kappa(\epsilon_\cX + r_\cX) + d$, the triangle inequality followed by (\ref{eq:rho'}) yields
\[
\|f(x) - \eta\|_{\R^m} < 2\kappa(\epsilon_\cX + r_\cX) + d + \rho' = 2\kappa(\epsilon_\cX+r_\cX) + r_\cY + 2d.
\]
Using (\ref{eq:d}), we finally obtain $\|f(x) - \eta\|_{\R^m} < \epsilon_\cY$. Thus, the intersection $\bigcap_{\eta \in \Delta_\sigma}\cB_{\epsilon_\cY}(\eta)$ is non-empty: it contains $f(x)$, and therefore determines a simplex of $N(Y')$. Any subset of $\Delta'_\sigma$ determines a face of this simplex and must also be a simplex in its own right.
\end{proof}

The proof of Theorem \ref{thm:finalresnoise} concludes with the observation that Propositions \ref{prop:nonempty-noise} and \ref{prop:gammasimp-noise} permit -- with probability larger than $(1-\delta_\cX)(1-\delta_\cY)$ -- the construction of a selector $h':X' \to Y'$ for $\Delta'$ which satisfies $h'(\xi) \in \Delta'(\xi)$. It is easy to check that $\iota_\cY \circ f \circ \pi_\cX\left(\cB_{\epsilon_\cX}(\xi)\right) \subset \cB_{\epsilon_\cY}(h'(\xi))$ for any such $h'$, so the induced simplicial map $\phi_{h'}:N(X') \to N(Y')$ is a simplicial reconstruction of $f$ as desired.

\section{A Proof of the Controlled Nerve Lemma}\label{sec:blocknerve}

In this section we prove Lemma \ref{lem:contner}, which was used to control the maps $\zeta_Z: |N(Z)| \to \cU(Z)$ for $Z \in \setof{X,Y}$ in the proof of Proposition \ref{prop:rechom}.  

\subsection{General Considerations}

The {\em mapping cylinder} $M(\gamma)$ of a continuous function $\gamma:\cU \to \cV$ between topological spaces is given by
\[
M(\gamma) = \frac{\left(\cU \times [0,1]\right) \coprod \cV}{(x,0) \sim \gamma(x)}.
\]
In other words, the mapping cylinder of $\gamma$ consists of the product of its domain $\cU$ with $[0,1]$ glued onto the co-domain $\cV$ via the identification of each $(x,1)$ with the image $\gamma(x)$. It is easy to see that there is a strong deformation retraction from $M(\gamma)$ to the codomain $\cV$ obtained by sliding each $(x,t) \in \cU \times [0,1]$ to the endpoint $\gamma(x)$. If $\gamma$ is a homotopy-equivalence, then $M(\gamma)$ also admits a strong deformation retraction to the domain $\cU \simeq \cU \times \setof{0} \subset M(\gamma)$.

\begin{defn}
{\em
A metrizable topological space $\cV$ is an {\em absolute neighborhood retract} (henceforth abbreviated $\ANR$) if for each triple $(\cU,\cA,\psi)$ of metrizable space $\cU$, closed subset $\cA \subset \cU$ and continuous function $\psi:\cA \to \cV$, there exists a pair $(\cA',\psi')$ consisting of an open set $\cA'$ with $\cA \subset \cA' \subset \cU$ and a map $\psi':\cA' \to \cV$ whose restriction $\psi'|_\cA$ to $\cA$ equals $\psi$.
}
\end{defn}

Throughout the sequel, every $\ANR$ is assumed to be {\em finite-dimensional} in the sense that it admits a continuous embedding into $\R^m$ for some $m \geq 0$. We will highlight all those basic facts about finite-dimensinal $\ANR$s which are relevant to our proofs and refer the reader to \cite[Ch.\ 11]{granas:dugundji} for details and proofs. Recall that a pair $(\cU,\cA)$ of topological spaces satisfies the {\em homotopy extension property} if for any triple $(\cV,\omega_t,\gamma)$ of topological space $\cV$, homotopy $\omega_t:\cA \to \cY$ and function $\gamma:\cU \to \cV$ with $\gamma|_{\cA} \equiv \omega_0$, there exist continuous functions $\omega_t':\cU \to \cV$ so that $\omega_0' \equiv \gamma$ on $\cU$ and $\omega'_t|_\cA \equiv \omega_t$ for each $t \in [0,1]$.

We make use of the following facts in our next lemma:
\begin{enumerate}
\item the mapping cylinder of a continuous function between $\ANR$s is also an $\ANR$, and
\item any pair $(\cU,\cA)$ of compact $\ANR$s with $\cA$ is closed in $\cU$ satisfies the homotopy extension property.
\end{enumerate} 

\begin{lem} 
\label{lem:mapcyl}
Let $\gamma:(\cU,\cA) \to (\cV,\cB)$ be a map of $\ANR$ pairs with $\cA$ and $\cB$ closed in $\cU$ and $\cV$, such that $\gamma:\cU \to \cV$ and $\gamma|_\cA: \cA \to \cB$ are homotopy-equivalences. Then, $\gamma$ is a homotopy-equivalence of pairs.
\end{lem}

\begin{proof} Let $\omega_t:M(\gamma|_\cA)\to M(\gamma|_\cA)$ be a strong deformation retraction from $M(\gamma|_\cA)$ to $\cA$. By the homotopy extension property, we can construct a map $\omega'_t:M(\gamma) \to M(\gamma)$ extending $\omega_t$ so that $\omega'_0|_\cU$ is the identity on $\cU$. Let $\nu_t:M(\gamma)\to M(\gamma)$ be a strong deformation retraction from $M(\gamma)$ to $\cU$. Define 
\begin{align*}
\theta_{t}= \begin{cases} 
								\omega'_{2t} & 0 \leq t \leq \nicefrac{1}{2} \\ 
								\nu_{2t-1}\circ \omega'_{1} & \nicefrac{1}{2} \leq t \leq 1
				\end{cases}.
\end{align*}
This is a strong deformation retraction of pairs from $(M(\gamma),\,M(\gamma|_\cA))$ to $(\cU,\cA)$. Therefore, $(\cU,\cA)$ is homotopy-equivalent as a pair to $(M(\gamma),\,M(\gamma|_\cA))$, which in turn is homotopy-equivalent as a pair to $(\cV,\cB)$.
\end{proof}

As a corollary of the proof, we see that any strong deformation retraction of $M(\gamma|_\cA)$ to $\cA$ can be extended to a strong deformation retraction of $M(\gamma)$ to $\cU$ with a new parameterization. The given strong deformation retraction over $\cA$ takes place in the first half of the interval [0,1] and remains stationary thereafter.  

\begin{prop} 
Let $\cL$ be a topological space, $N$ a finite simplicial complex and $\rho:\cL \to |N|$ a continuous surjection. If $\cL^\sigma := \rho^{-1}(|\sigma|)$ is a contractible $\ANR$ for each simplex $\sigma \in N$, then there is a strong deformation retraction from $M(\rho)$ to $\cL$ whose restriction  to the mapping cylinder of $\rho|_{\cL^\sigma}:\cL^\sigma \to |\sigma|$ is a strong deformation retraction onto $\cL^\sigma$ for each $\sigma \in N$.
\end{prop}

\begin{proof} We proceed by induction on skeleta of $N$ using Lemma \ref{lem:mapcyl}. Pick any $0$-dimensional simplex $\sigma \in N^0$. Now, $\rho|_{\cU^\sigma}:\cU^\sigma \to |\sigma|$ is a homotopy-equivalence by assumption, and therefore $\cU^\sigma$ is a strong deformation retract of the mapping cylinder $M(\rho|_{\cU^\sigma})$. The disjoint union of such strong deformation retractions as $\sigma$ ranges over $N^0$ yields a strong deformation retraction over all $N^0$, thus establishing the base case of our induction. Assume the inductive hypothesis for the $(d-1)$-skeleton of $N$ for some $d \geq 1$. For each $\sigma \in N^d$, we have a strong deformation retraction from the mapping cylinder $M(\rho|_{\cL^\sigma})$ to $\cL^\sigma$ since the restriction of $\rho$ to $\cL^\sigma$ is a homotopy-equivalence. On the other hand, the inductive hypothesis guarantees a strong deformation retraction from the mapping cylinder $M(\rho|_{\partial\sigma})$ to $\partial\sigma$, where $\partial\sigma$ denotes the boundary of $\sigma$ in $N$. Lemma \ref{lem:mapcyl} now gives us a strong deformation retraction from $M(\rho|_{\cL^\sigma})$ to $\cL^\sigma$ extending the strong deformation retraction over $\partial\sigma$. Piecing these extensions together for all $\sigma \in N^d$ gives the desired result.
\end{proof}

The following corollary is immediate; it plays a central role in our proof of the controlled nerve lemma.

\begin{cor}
\label{cor:blocks}
Let $\cL$ be a topological space, $N$ a finite simplicial complex and $\rho:\cL \to |N|$ a continuous surjection. If $\rho^{-1}(|\sigma|)$ is a contractible $\ANR$ for each $\sigma \in N$, then $\rho$ admits a homotopy-inverse $\iota:|N| \to \cL$ so that $\iota(|\sigma|) \subset \rho^{-1}(|\sigma|)$ for each $\sigma \in N$.
\end{cor}

\subsection{Specific Considerations}

Let $\U$ be a finite collection of open balls in Euclidean space $\R^n$ whose union is denoted by $\cU$. Let $N(\U)$ be the nerve of the cover of $\cU$ by the balls in $\U$. We remark that a finite union of balls in $\R^n$ is an $\ANR$, and so is the geometric realization of a finite simplicial complex. Note also that $\ANR$s are closed under the operation of taking finite products. 

The following consequence of the main theorem from \cite{viet-smale} will be used frequently in order to establish homotopy-equivalences.
\begin{thm}[\bf The Vietoris-Smale Theorem]
Let $\gamma:\cU \to \cV$ be a proper surjective continuous map of connected, locally contractible separable metric spaces. If the fiber $\gamma^{-1}(v) \subset \cU$ is contractible for each $v \in \cV$, then the map induced by $\gamma$ on the homotopy groups of $\cU$ and $\cV$ is an isomorphism.  
\end{thm}

In order to use this theorem, one must construct proper functions\footnote{Recall that $\gamma:\cU \to \cV$ is {\em proper} if $\gamma^{-1}(K) \subset \cV$ is compact whenever $K \subset \cU$ is compact.} and hence deal with closed rather than open balls. 
\begin{rem} 
{\em
We may modify $\U$ into a finite collection of closed balls without altering the homotopy type of $N(\U)$. For each top-dimensional simplex $\sigma \in N(\U)$, fix a witness point $x_\sigma \in \supp{\sigma}$ and note by openness that there exists some $d_\sigma > 0$ so that the ball of radius $d_\sigma$ around $x_\sigma$ lies entirely in $\supp{\sigma}$. Now, set 
\[
d_\U = \min_\sigma d_\sigma,
\] 
where $\sigma$ ranges over all top-dimensional simplices in $N(\U)$. Pick any $d \in (0,d_\U)$ and replace each open ball $u \in \U$ by a closed ball $u'$ with the same center and with radius shrunk by $d$. This modification preserves the non-emptiness of all intersections. The contractibility of the new closed intersections follows immediately from the fact that they are convex sets.
}
\end{rem}

In light of the preceding remark, we assume that $\U$ is a finite collection of closed balls in $\R^n$ and hence that the union $\cU$ is also closed. Our goal is to establish the existence of a homotopy-equivalence $\zeta:N(\U) \to \cU$ so that $\zeta(|\sigma|)$ is contained in the union of balls $\capp{\sigma}$. To this end, consider the subset $\cL$ of the product $|N(\U)| \times \cU$ defined as follows. We include $(\alpha,x)$ in $\cL$ if and only if there exists some simplex $\sigma \in N(\U)$ with $\alpha \in |\sigma|$ and $x \in \supp{\sigma}$. 

We now have surjective, continuous and proper projection maps from $\cL$ onto each of the constituent factors of the product:
\[
\begin{diagram}
\dgARROWLENGTH 0.5in
\node{}
\node{\cL}\arrow{sw,t}{\rho_1} \arrow{se,t}{\rho_2}
\node{}
\\
\node{|N(\U)|}
\node{}
\node{M}
\end{diagram}
\]

\begin{prop}
The maps $\rho_1$ and $\rho_2$ are homotopy-equivalences.
\end{prop}

\begin{proof}
Since all spaces in sight certainly satisfy the requirements of the Vietoris-Smale theorem, it suffices to check that the fibers of $\rho_1$ and $\rho_2$ are contractible. First consider the fiber $\rho_1^{-1}(\alpha)$ over some barycentric function $\alpha \in |N(\U)|$. By finiteness of the cover $\U$, there is a unique simplex $\sigma_\alpha \in N(\cU)$ of minimal dimension so that $\alpha \in |\sigma_\alpha|$. By definition, we have
\begin{align*}
\rho_1^{-1}(\alpha) &= \setof{(\alpha,x) \mid x \in \supp{\sigma} \text{ for some }\sigma \supset \sigma_\alpha}\\
									 &= \setof{\alpha} \times \bigcup_{\sigma \supset \sigma_\alpha} \supp{\sigma}\\
									 &= \setof{\alpha} \times \supp{\sigma_\alpha}.
\end{align*}
The second factor $\supp{\sigma_\alpha}$ is contractible, being an intersection of convex sets -- and therefore, so is the fiber $\rho_1^{-1}(\alpha)$. Applying the Vietoris-Smale theorem concludes this half of the proof.

On the other hand, pick any $x \in \cU$ and consider the fiber $\rho_2^{-1}(x) \subset \cL$. Since the cover $\U$ is finite by assumption, there is a unique simplex $\sigma_x \in N(\U)$ of maximal dimension satisfying $x \in \supp{\sigma_x}$. By definition, we have
\begin{align*}
\rho_2^{-1}(x) &= \setof{(\alpha,x) \in |\sigma| \times \setof{x} \mid \sigma \in N(\U) \text{ satisfies }x \in \supp{\sigma}} \\
						  &= |\sigma_x| \times \setof{x}
\end{align*} 
Since the realization of any simplex is contractible, the fiber $\rho_2^{-1}(x)$ is the product of contractible sets and hence is contractible itself. By the Vietoris-Smale theorem, $\rho_2$ induces an isomorphism of homotopy groups.
\end{proof}

In order to apply Corollary \ref{cor:blocks}, we re-examine the fibers of $\rho_1$. For each $\sigma \in N(\U)$ one obtains the following: 
\begin{align*}
\rho_1^{-1}(|\sigma|) &= \setof{(\alpha,x) \mid \exists \sigma' \subset \sigma \text{ with }\alpha \in |\sigma'| \text{ and }x \in \supp{\sigma'}} \\
								&= |\sigma| \times \bigcup_{\sigma' \subset \sigma}\supp{\sigma'}\\
								&= |\sigma| \times \capp{\sigma}.
\end{align*} 
The second factor $\capp{\sigma}$ is contractible, being a union of convex sets (i.e., balls in Euclidean space) with a non-empty intersection. Now, $|\sigma| \times \capp{\sigma}$ is a product of contractible $\ANR$s, and hence a contractible $\ANR$. Therefore, $\rho_1$ satisfies the hypotheses of Corollary \ref{cor:blocks} and hence admits a homotopy-inverse $\iota:|N(\U)| \to \cL$ so that 
\[
\iota(|\sigma|) \subset \rho_1^{-1}(|\sigma|) = |\sigma| \times \capp{\sigma}.
\]
Letting $\zeta:|N(\U)| \to \cU$ be the composition $\rho_2\circ\iota$ of homotopy-equivalences gives the desired inclusion $\zeta(|\sigma|) \subset \capp{\sigma}$ and concludes the proof of Lemma \ref{lem:contner}.


%
%
\bibliography{mh}
\bibliographystyle{abbrv}
\end{document}